\documentclass[11pt]{amsart}
\usepackage{amsthm}
\usepackage{amsfonts}
\usepackage{amssymb}
\usepackage{amsmath}
\usepackage{mathrsfs}
\usepackage{xcolor}
\usepackage{hyperref}
\usepackage[utf8]{inputenc}   
\usepackage[T1]{fontenc}    
\usepackage{cite}
\begin{document}

\large


\theoremstyle{plain}
\newtheorem{theorem}{Theorem}[section]

\newtheorem{prop}[theorem]{Proposition}

\newtheorem{lemma}{Lemma}[section]
\newtheorem{corollary}{Corollary}

\theoremstyle{definition}
\newtheorem{definition}{Definition}[section]
\newtheorem{remark}{\textnormal{\textbf{Remark}}}
\newtheorem{proposition}{\textnormal{\textbf{Proposition}}}
\newtheorem*{acknowledgement}{\textnormal{\textbf{Acknowledgement}}}
\theoremstyle{remark}
\newtheorem{example}{Example}
\numberwithin{equation}{section}
\newtheorem*{question}{\textnormal{\textbf{Open question}}}

\title[Unlocking Novel Topological Structures via Rough Families]{Unlocking Novel Topological Structures via Rough Families}

\author[Sourav Mandal, Lakshmi Kanta Dey, Pratikshan Mondal]{Sourav Mandal$^1$, Lakshmi Kanta Dey$^{1,*}$, Pratikshan Mondal$^2$}

\thanks{$^*$ Corresponding author}

\address{$^1$Department of Mathematics, National Institute of Technology Durgapur, 713209, West Bengal, India.}

\address{$^2$Department of Mathematics, Acharya Brojendra Nath Seal College, Cooch Behar, 736101, West Bengal, India.}

 \thanks{\emph{E-mail}: souragence@gmail.com (S. Mandal), lakshmikdey@yahoo.co.in (L. K. Dey), real.analysis77@gmail.com (P. Mondal)}
 \thanks{ORCID: 0000-0002-6703-1958 (S. Mandal), 0000-0001-5389-6048 (L. K. DEY), 0000-0002-9678-9610 (P. Mondal)}
\date{}

\subjclass[2020]{54A05, 54A20, 54C08, 54D05, 54D30}

\keywords{Rough family, rough open set, rough topology, rough homeomorphism, rough compactness, rough connectedness.}

\date{}

\thanks{Research of the first author is supported by the
Council of Scientific and Industrial Research, Government of India (File No:
09/0973(25105)/2025-EMR-I)}

\setcounter {page}{1}

\maketitle

\begin{abstract}
    
Very recently, the notion of rough family has been introduced in [Leonetti, P., J. Convex Anal. 32(4):1083-1090, 2025] to explore rough ideal convergence in topological spaces where the limit of a sequence may not be unique. This raises the question of whether $T_2$ topological spaces can be characterized using rough families. In this article, we prove that a topological space is $T_2$ if and only if it can never be a rough topological space. In this context, we first introduce the notions of rough interior and rough closure of a set from the perspective of a rough family, which leads to the definition of rough open sets (rough closed sets). As a consequence, we generate a new topology, termed rough topology, as well as rough homeomorphism. Our main contribution presents the novelty of this new class; in particular, we explicitly construct several examples which ensure that two non-homeomorphic spaces can be roughly homeomorphic under certain roughness. Additionally, we extend the concepts of compactness as well as connectedness, where our findings diverge from existing literature in these areas, in a nutshell, providing new insights and perspectives.

\end{abstract}

\section{Introduction}	   
     		
    In 2001, Phu \cite{Phu2001} proposed a new concept called rough convergence as a generalization of the classical notion of convergence in a normed linear space. After that, various mathematicians expanded this notion in several directions, see e.g. \cite{Aytar2008(2), Aytar2008(1), Pal2013, Phu2003, Rahaman} and references therein. After a few gaps, in 2021 Ghosal et al. \cite{SGSM1, SGSM2} introduced the concept of rough weighted ideal convergence in locally solid Riesz spaces. Meanwhile, in 2025, Leonetti \cite{Leonetti2023} proposed the idea of rough ideal convergence specifically in topological spaces by employing a rough family. In their work, Ghosal et al.~\cite{SGSM1, SGSM2} and Leonetti~\cite{Leonetti2023} both replaced the roughness degree with a neighborhood structure instead of a nonnegative real number. Here, we present the definition of rough ideal convergence proposed by Leonetti \cite{Leonetti2023}.

    \begin{definition}[\cite{Leonetti2023}, Definition 1.1]
         Let $x=\{x_n\}_{n\in \mathbb{N}}$ be a sequence in a topological space $(X, \tau).$ Also, let $\mathscr{F}= \{F_\eta: \eta\in X\}$ be a rough family, i.e., a collection of subsets of $X$ with the property that $\eta \in F_\eta$ for all $\eta\in X.$ Then $x=\{x_n\}_{n\in \mathbb{N}}$ is said to be $\mathcal{I}$-convergent to $\eta \in X$ with roughness $\mathscr{F},$ shortened as $(\mathcal{I}, \mathscr{F}, \tau)$-$\lim x_n = \eta,$ if $\{n \in \omega : x_n \notin U\}\in\mathcal{I}$ for all $\tau$-open sets $U \subseteq X$ such that $F_\eta\subseteq U$ where $\omega$ denotes the set of all nonnegative integers.              
    \end{definition}
    For $\mathcal{I}:=\mathcal{I}_{fin}=\{A\subset\mathbb{N}: A~\text{is a finite set}\}$, $\mathcal{I}$-convergence with roughness $\mathscr{F}$ corresponds to the conventional rough convergence in a topological space. Furthermore, when $\mathscr{F}:=\mathscr{F^\#}=\{\{\eta\}: \eta\in X\}$, rough convergence simplifies to the usual notion of convergence in topological spaces. For a comprehensive understanding of these areas, we refer to \cite{Leonetti2023, Leonetti2026} and the references cited therein. From now towards, we consider $\mathscr{F}=\{F_x: x\in X\}$, $\mathscr{R}=\{R_x: x\in X\}$, and $\mathscr{L}=\{L_y:y\in Y\}$ as rough families in various spaces as required. Usually, it is assumed that $\mathscr{F} \neq \mathscr{F^\#}$ unless stated otherwise. We write "$\mathscr{F}\leq \mathscr{R}$" to signify that for all $x\in X$, $F_x\subset R_x$. Thus, $\mathscr{F}^\#\leq \mathscr{F}$ holds for any rough family $\mathscr{F}$.\\

    Now, for $\mathscr{F}\neq \mathscr{F^\#}$, the limit of a sequence may not be unique. This raises the question of whether $T_2$ topological spaces can be characterized using rough families. The primary objective of this study is to give a positive answer to this query. Moreover, our aim to explore whether the notion of rough families can be used to generate novel topological concepts that extend existing notions. To achieve this, the paper is organized into four sections. In Section \ref{Sec1}, we first introduce the concepts of rough interior and rough closure of a set. In general, rough closure may not be closed; likewise, rough interior may not be open as well. Also, rough closure always contains more than one element under certain roughness. Furthermore, Section \ref{Sec1} presents the ideas of rough open sets and rough closed sets. 
    Concurrently, we propose the notion of rough continuity in Section \ref{Sec3} and establish that continuity is a particular instance of rough continuity. Moreover, the notion of rough homeomorphism has been introduced to show that two non-homeomorphic spaces may be equivalent in the sense of rough homeomorphism. In Section \ref{Sec4}, based on the preceding notion, we develop a new topology, termed a rough topology, within a given topological space. Although the rough topology is coarser than the original topology, it satisfies additional topological properties that the original topology does not. We prove that a topological space is $T_2$ if and only if it can never be a rough topological space. We provide an approach using rough homeomorphism and rough topology to establish the homeomorphism between two spaces. Finally, in Section \ref{Sec5}, we extend the concepts of compactness and connectedness to generalized notions, namely rough compactness and rough connectedness. Numerous established results regarding compactness fail to apply within the domain of rough compactness. In our research, we have identified the corresponding outcomes that offer fresh insights and perspectives.

\section{Rough closed set and rough open sets with roughness $\mathscr{F}$}\label{Sec1}

In this section, we present two new concepts, identified as rough closed sets and rough open sets by employing a rough family \(\mathscr{F}\), which are formulated based on the notions of rough closure and rough interior of a set. Remember that for a set $A$ in a topological space $(X,\tau)$, the closure of \( A \), is \( \overline{A}=A\cup \{x \in X : \text{for any}~U \in \tau,~\text{where}~x \in U, U \cap A \setminus \{x\} \neq \emptyset\}\). Similarly, the interior of \( A \) is \( A^o =\{x \in A : \text{there exists}~U \in \tau~\text{such that}~x \in U \subset A\}\). Analogously, the definitions of rough closure and rough interior with roughness $\mathscr{F}$ are as follows:

\begin{definition}
 Let $\mathscr{F}$ be a rough family in a topological space $(X,\tau)$. An element $x\in X$ is defined as a rough limit point of a subset $A\subset X$ with roughness $\mathscr{F}$ if, for any open set $U$ in $(X,\tau)$ with $F_x\subset U$, $U\cap A\setminus\{x\}\neq \emptyset$. The rough closure with roughness $\mathscr{F}$ of $A$ is defined by 
$\overline{A}^\mathscr{F}=A\cup A^{'\mathscr{F}}$, where $A^{'\mathscr{F}}$, termed the rough derived set of $A$ with roughness $\mathscr{F}$, is defined by $$A^{'\mathscr{F}}=\{x\in X: x~ \text{is a rough limit point of }A~\text{with roughness}~\mathscr{F}\}.$$
  
    Similarly, an element $x\in A$ is called a rough interior point of $A$ with roughness $\mathscr{F}$ if there exists an open set $U$ in $(X,\tau)$ with $F_x\subset U$ and $U\subset A.$ Consequently, the rough interior of $A$ with roughness $\mathscr{F}$ is defined by 
    $$A^{o\mathscr{F}}=\{x\in A: x~\text{is a rough interior point of}~A~\text{with roughness}~\mathscr{F}\}.$$
\end{definition}

\begin{remark}
    For any sequence $x=\{x_n\}_{n\in\mathbb{N}}$ in $A\subset X$, it is immediate that $$Lim_x(\mathcal{I}_{fin}, \mathscr{F})\subseteq\Gamma_x(\mathcal{I}_{fin},\mathscr{F})\subseteq \overline{A}^\mathscr{F},$$ where the sets $Lim_x(\mathcal{I}_{fin}, \mathscr{F})$ and $\Gamma_x(\mathcal{I}_{fin},\mathscr{F})$ are the set of all rough ideal limits and rough ideal cluster points of $x=\{x_n\}_{n\in\mathbb{N}}$ with roughness $\mathscr{F}$ and ideal $\mathcal{I}:=\mathcal{I}_{fin}$, as defined in \cite{Leonetti2026}.
\end{remark}

The relationships among the closure, interior, rough closure, and rough interior of a set are described in the following theorem.

\begin{theorem}\label{T1} Let $A\subset X$ and $\mathscr{F}$ be a rough family in $X$. Then, $$A^{o\mathscr{F}}\subset A^{o}\subset A\subset \overline{A}\subset \overline{A}^\mathscr{F}.$$ Moreover, if $\mathscr{R}$ be another rough family in $X$ such that $\mathscr{F}\leq \mathscr{R}$. Then, $$A^{o\mathscr{R}}\subset A^{o\mathscr{F}} \subset A\subset \overline{A}^\mathscr{F}\subset \overline{A}^\mathscr{R}.$$ 

\end{theorem}

\begin{proof}
     Let $x\in A^{o\mathscr{F}}.$ Then, there exists an open set $U$ in $X,$ such that $F_x\subset U$ and $U\subset A.$ Hence, $x\in U\subset A$ and this implies that $A^{o\mathscr{F}}\subset A^{o}\subset A.$ 

     For the other part, let us consider $x\in \overline{A}.$ Let $U$ be any open set in $X$ with $F_x\subset U.$ Consequently, $x\in U$ and thus $U\cap A\setminus\{x\}\neq \emptyset.$ Hence, $A\subset \overline{A}\subset \overline{A}^\mathscr{F}.$  

     The rest of the theorem can be proved similarly to the proof above.
\end{proof} 

In general, for $A\subset X,$ the sets $\overline{A}^\mathscr{F}$ and $A^{o\mathscr{F}}$ may not necessarily be closed or open. This is demonstrated by the following example:

\begin{example}\label{Ex2}
    Let us consider the sets $A=[1,5]$ and $B=(1,5)$ in $\mathbb{R}$ equipped with the usual topology i.e, in $\mathbb{R}_u$. Construct the rough family $$\mathscr{F}=\{(x-0.1,x+0.1): x\in \mathbb{R}\}.$$ Consequently, we obtain 
    $$\overline{A}^\mathscr{F}=(0.9,5.1)~~~\text{and}~~~B^{o\mathscr{F}}=[1.1,4.9].$$
    Clearly, the set $\overline{A}^\mathscr{F}$ is not a closed set and $B^{o\mathscr{F}}$ is not an open set.
\end{example}

Now we will examine the conditions under which the set $\overline{A}^\mathscr{F}$ is closed and the set $A^{o\mathscr{F}}$ is open. In this context, we recall the upper Vietoris topology $\hat{\tau}$ on $X$. For a topological space $(X,\tau)$, the upper Vietoris topology $\hat{\tau}$ is generated by the basis of sets $\{F\in \mathcal{H}(X): F\subseteq U, ~U\in \tau\}$, where $\mathcal{H}(X)=\{F\subseteq X: F~\text{is nonempty and closed}\}$.

\begin{theorem}\label{Th2.2} Let $A$ be any subset of a topological space $(X,\tau)$ with a  rough family $\mathscr{F}$. Then, the following statements hold: 
   \begin{itemize}
       \item[(i)] The set $\overline{A}^\mathscr{F}$ is closed if $F_x$ is closed for each $F_x\in \mathscr{F}$ and the map $x\to F_x$ is $\hat{\tau}$-continuous.

       \item[(ii)]The set $A^{o\mathscr{F}}$ is open if $F_x$ is open for each $F_x\in \mathscr{F}$ and $y\in F_x$ implies that $F_y\subset F_x$.
   \end{itemize} 
\end{theorem}

\begin{proof} 
\begin{itemize}
    \item[(i)] If $\overline{A}^\mathscr{F}=\emptyset,$ the proof is over. Therefore, suppose that $\overline{A}^\mathscr{F}\neq\emptyset$. Now, let us consider a net $\{x_i\}_{i\in I}$ in $\overline{A}^\mathscr{F}$ such that $x_i\to x\in X.$ Let $U$ be an open set in $X$ with $F_x\subset U.$ Now, due to the $\hat{\tau}$-continuity of $x\to F_x,$ $F_{x_i}\to F_x.$ Thus there exists a $j\in I$ so that $F_{x_j}\in \hat{U}$ or equivalently, $F_{x_j}\subset U$ where $\hat{U}=\{F\in \mathcal{H}(X): F\subset U\}.$  Now since $x_j\in \overline{A}^\mathscr{F},$ it directly implies that $x\in \overline{A}^\mathscr{F}.$ 
    
    \item[(ii)] Let, $x\in A^{o\mathscr{F}}.$ Then there is $U\in \tau$ so that $x\in F_x\subset U\subset A$. Now, for any $y\in F_x$, $F_y\subset F_x\subset U\subset A$ which implies $y\in A^{o\mathscr{F}}$ or $F_x\subset A^{o\mathscr{F}}$. Consequently, $x\in F_x\subset A^{o\mathscr{F}}$ and $F_x$ is open. Hence, $x\in (A^{o\mathscr{F}})^o$ and thus $A^{o\mathscr{F}}$ becomes open.
\end{itemize}
\end{proof}

Note that the condition (ii) in Theorem \ref{Th2.2} is sufficient but not necessary. This is clearly illustrated by Example \ref{Ex3}. We shall now examine the set-theoretic properties of rough closure and rough interior.

\begin{theorem}
    Let $\mathscr{F}$ be a rough family in a topological space $(X,\tau)$. Then, for $A, ~B\subset X,$ the following statements are true:
    \begin{itemize}
        \item[(i)] If $A\subset B,$ then $\overline{A}^\mathscr{F}\subseteq \overline{B}^\mathscr{F},$ ~$A^{o\mathscr{F}}\subseteq B^{o\mathscr{F}}.$
        
        \item[(ii)] $\overline{A}^\mathscr{F}\cup\overline{B}^\mathscr{F}= \overline{(A\cup B)}^\mathscr{F}$ and $A^{o\mathscr{F}}\cup B^{o\mathscr{F}}\subset (A\cup B)^{o\mathscr{F}}.$ In the second case, equality might not happen.
        
        \item[(iii)] $\overline{(A\cap B)}^\mathscr{F}\subset \overline{A}^\mathscr{F}\cap\overline{B}^\mathscr{F}$ and $(A\cap B)^{o\mathscr{F}}=A^{o\mathscr{F}}\cap B^{o\mathscr{F}}.$ Usually, the equality might not be true for the first case.
        
        \item[(iv)] $\left(\overline{A}^\mathscr{F}\right)^c=(A^c)^{o\mathscr{F}},$ $\left(A^{o\mathscr{F}}\right)^c=\overline{(A^c)}^\mathscr{F}.$
    \end{itemize}
\end{theorem}

\begin{proof}
 \begin{itemize}
     \item[(i)]The proof follows directly from the definitions.
     
     \item[(ii)] The first part of (ii) follows from the definitions of rough closure jointly with (i). 
     
     For the subsequent part, let us consider $\mathbb{R}$ endowed with the lower limit topology $\mathbb{R}_l.$ Let $A=[0,1)$ and $B=[1,2).$ Then for the rough family $\mathscr{F}=\left\{\left\{x,2x\right\}: x\in \mathbb{R}\right\},$ we obtain $$A^{o\mathscr{F}}\cup B^{o\mathscr{F}}=\left[0,\frac{1}{2}\right)\cup\emptyset=\left[0,\frac{1}{2}\right)\neq \left[0,1\right)=(A\cup B)^{o\mathscr{F}}.$$
     
     \item[(iii)]  The first part of (iii) is also a consequence of (i) and the definition of rough interior.
     
     For the later part, we substitute the rough family $\mathscr{F}$ with $\mathscr{F}=\left\{\left\{x,\frac{x}{2}\right\}: x\in \mathbb{R}\right\}$ in the aforementioned scenario, yielding $$\overline{(A\cap B)}^\mathscr{F}=\overline{\emptyset}^\mathscr{F}=\emptyset\neq [1,2)=[0,2)\cap [1,4)=\overline{A}^\mathscr{F}\cap\overline{B}^\mathscr{F}.$$
     
     \item[(iv)]  For the first part, note that $x\in (A^c)^{o\mathscr{F}}$ if and only if there is $U\in \tau$ with $F_x\subset U\subset A^c,$ i.e., $(U\cap A)\setminus\{x\}=\emptyset$ which implies that $x\notin \overline{A}^\mathscr{F}$ or $x\in \left(\overline{A}^\mathscr{F}\right)^c$.

     For the second part, let $x\in \left(A^{o\mathscr{F}}\right)^c.$ Then, there exists $U\in \tau$ with $F_x\subset U$ but $\left(U\cap A^c\right)\setminus\{x\}\neq \emptyset$. Thus, $x\in \overline{(A^c)}^\mathscr{F}$ implies $\left(A^{o\mathscr{F}}\right)^c\subset \overline{(A^c)}^\mathscr{F}$. Similarly, we can obtain $\overline{(A^c)}^\mathscr{F}\subset \left(A^{o\mathscr{F}}\right)^c$ and we are done.
 \end{itemize}
\end{proof}

In a $T_1$ space, every singleton set $A$ is closed, and consequently, $\overline{A}=A$. However, in this context, if $\mathscr{F} \neq \mathscr{F^\#}$, then $\overline{A}^{\mathscr{F}}$ invariably contains more than one point, as demonstrated in the following theorem:

\begin{theorem}
    Let $f: X\to X$ be a bijective map such that $f(x)\neq x$ for all $x\in X$. Then, for the rough family, $\mathscr{F}=\{F_x=\{x,f(x)\}: x\in X\}$, $\overline{A}^{\mathscr{F}}$ contains more that one point for any $A\subset X$.
\end{theorem}

\begin{proof}
    If $A\subset X$ contains more than one element, then there is nothing to prove. So, consider $A=\{a\}$ for some $a\in X$. Since, $f: X\to X$ is bijective and $f(x)\neq x$ for all $x\in X$, there exists $b\in X$ such that $b\neq a$ and $f^{-1}(a)=b$. Now, $F_b=\{b,f(b)\}=\{b,a\}$. Hence, for any open set $U$ in $X$ with $F_b\subset U$, $a\in (U\cap A)\setminus \{b\}$ implies $(U\cap A)\setminus \{b\}\neq \emptyset$. Consequently, $\{a,b\}\subset \overline{A}^{\mathscr{F}}$.
\end{proof}

\begin{corollary}
    Moreover, if $X$ is an ordered topological space in the above theorem, then $[a,b]\subset  \overline{A}^\mathscr{F}$ or $[b,a]\subset  \overline{A}^\mathscr{F}$ where $b=f^{-1}(a)$.
\end{corollary}

\begin{proof}
    Since, $f(x)\neq x$ for all $x\in X$, so either $a<b$ or $a>b$. Without loss of generality, let $a>b$. Moreover, let $c\in [b, a]$. Then for any open set $U$ in $X$ with $F_b=\{b,a\}$  is of the form $(x,y)$ where $x,y\in X$ satisfying $x<b<a<y$. Consequently, $c\in U$ and $(U\cap A)\setminus \{c\}\neq \emptyset$ imply $c\in \overline{A}\subset \overline{A}^\mathscr{F}$. 
\end{proof}

Our next aim is to determine the rough closure and rough interior of a set within a finite product of topological spaces.

\begin{theorem}\label{Th2.5}
     Let $\mathscr{F}=\{F_x: x\in X\}$ and $\mathscr{L}=\{L_y: y\in Y\}$ be two rough families in two topological spaces \((X,\tau)\) and \((Y,\zeta)\) respectively. Then for $A\subset X$, $B\subset Y$, the following are true 
     \begin{itemize}
         \item[(i)] $\overline{A\times B}^{\mathcal{G}}=\overline{A}^\mathscr{F}\times \overline{B}^{\mathscr{L}};$
         \item[(ii)] $(A\times B)^{o\mathcal{G}}=A^{o\mathscr{F}}\times B^{o\mathscr{L}}$. 
     \end{itemize}
     where the rough family $\mathcal{G}$ in $X\times Y$ is defined by $\mathcal{G}=\{G_{(x,y)}=F_x\times L_y: F_x\in \mathscr{F}~\text{and}~L_y\in \mathscr{L}\}$.
\end{theorem}

\begin{proof}\begin{itemize}
         \item[(i)]    Let $(x,y)\in \overline{A\times B}^{\mathcal{G}}$. Let $U\in\tau$ and $V\in \zeta$ such that $F_x\subset U$ and $L_y\subset V$. Then $U\times V$ is an open set in $X\times Y$ with $G_{(x,y)}=F_x\times L_y\subset U\times V$. Consequently, we obtain
    $$\left\{(U\times V)\cap (A\times B)\right\}\setminus \{(x,y)\}\neq \emptyset,$$
    which directly implies, $$(U\cap A)\setminus\{x\}\neq\emptyset~\text{and}~(V\cap B)\setminus\{y\}\neq \emptyset.$$
    Hence, we have $x\in \overline{A}^\mathscr{F}$ and $y\in \overline{B}^{\mathscr{L}}$ or equivalently, $(x,y)\in \overline{A}^\mathscr{F}\times \overline{B}^{\mathscr{L}}$.\\

    Conversely, let  $(x,y)\in \overline{A}^\mathscr{F}\times \overline{B}^{\mathscr{L}}$ and let $W$ be an open set in $X\times Y$ with $G_{(x,y)}=F_x\times L_y\subset W$. Then, for any $(z,t)\in F_x\times L_y$, there exist $U_z\in \tau$, $V_t\in \zeta$ such that $(z,t)\in U_z\times L_y\subset W$ and $(z,t)\in F_x\times V_t\subset W$. Let $U=\bigcup_{z\in F_x} U_z$ and $V=\bigcup_{t\in L_y} V_t$. Consequently,  $U\in \tau$ with $F_x\subset U$, $V\in \zeta$ with $L_y\subset V$, and  $U\times V\subset W$. Now $x\in \overline{A}^\mathscr{F}$, $y\in \overline{B}^{\mathscr{L}}$ imply $$(U\cap A)\setminus\{x\}\neq\emptyset~\text{and}~(V\cap B)\setminus\{y\}\neq \emptyset.$$
    Hence, $$\left\{W\cap (A\times B)\right\}\setminus \{(x,y)\}\neq \emptyset ~~~\text{implies}~~~(x,y)\in \overline{A\times B}^{\mathcal{G}}.$$ Subsequently, we obtain  $\overline{A\times B}^{\mathcal{G}}=\overline{A}^\mathscr{F}\times \overline{B}^{\mathscr{L}}$.\\
    \item[(ii)]   Let $(x,y)\in (A\times B)^{o\mathcal{G}}$. Then, there is an open set $W$ in $X\times Y$ with $$G_{(x,y)}=F_x\times L_y\subset W\subset A\times B.$$ Then as above, we get $U\in \tau$, $V\in \zeta$ such that $F_x\times L_y\subset U\times V\subset W\subset A\times B$. In other words, we can write $F_x\subset U\subset A$ and $L_y\subset V\subset B$. Therefore, $(x,y)\in A^{o\mathscr{F}}\times B^{o\mathscr{L}}$ and it follows that $(A\times B)^{o\mathcal{G}}\subset A^{o\mathscr{F}}\times B^{o\mathscr{L}}$.

    Conversely, let $(x,y)\in A^{o\mathscr{F}}\times B^{o\mathscr{L}}$. Then there are $U\in \tau$, $V\in \zeta$ with $F_x\subset U\subset A$ and $L_y\subset V\subset B$. Consequently, $G_{(x,y)}=F_x\times L_y\subset U\times V \subset A\times B$ implies $(x,y)\in (A\times B)^{o\mathcal{G}}$.
     \end{itemize}
\end{proof}

If the rough family $\mathcal{G}$ on $X\times Y$ is defined differently, we might get different results from Theorem \ref{Th2.5}. Here is an example:

\begin{example}
    Let us consider both the topological spaces $(X,\tau)$ and $(Y,\zeta)$ as $\mathbb{R}_u$. We construct the rough families $\mathscr{F}$ on $X$, $\mathscr{L}$ on $Y$, and $\mathcal{G}$ on $X\times Y$ as follows:
    \begin{multline*}
     \mathscr{F}=\mathscr{L}=\left\{(x-1,x+1):x\in \mathbb{R}\right\}~~~\text{and}~~~\\\mathcal{G}=\left\{\left(x-\frac{1}{2},x+\frac{1}{2}\right)\times\left(y-\frac{1}{2},y+\frac{1}{2}\right): (x,y)\in X\times Y\right\}.
    \end{multline*}
    Now, let $A=(0,5)$ and $B=(1,6)$. Then, $$\overline{A}^\mathscr{F}=(-1,6),~\overline{B}^\mathscr{L}=(0,7)~\text{and thus}~ \overline{A}^\mathscr{F}\times \overline{B}^\mathscr{L}=(-1,6)\times(0,7).$$
    Clearly, \begin{multline*}
        \left(-\frac{1}{2},\frac{1}{2}\right)\in \overline{A}^\mathscr{F}\times \overline{B}^\mathscr{L}~\text{but}~\left(-\frac{1}{2},\frac{1}{2}\right)\notin \overline{(A\times B)}^\mathcal{G}\\~\text{as}~\left[\left(-1,0\right)\times (0,1)\right]\bigcap (A\times B)\setminus\left\{\left(-\frac{1}{2},\frac{1}{2}\right)\right\}=\emptyset.
    \end{multline*}
    Similarly, 
    $$A^{o\mathscr{F}}=[1,4],~B^{o\mathscr{L}}=[2,5]~\text{but}~(A\times B)^{o\mathcal{G}}=
    \left[\frac{1}{2},\frac{9}{2}\right]\times \left[\frac{3}{2},\frac{11}{2}\right].$$
    Subsequently, $\overline{(A\times B)}^\mathcal{G}\neq \overline{A}^\mathscr{F}\times \overline{B}^\mathscr{L}$ and $(A\times B)^{o\mathcal{G}}\neq A^{o\mathscr{F}}\times B^{o\mathscr{L}}$.
\end{example}

\begin{remark}
    In Theorem \ref{Th2.5}, let $\mathcal{G}$ be a rough family in $X\times Y$ such that \begin{itemize}
        \item[(i)] $G_{(x,y)}\subset F_x\times L_y$ for all $(x,y)\in X\times Y$, then 
    $$\overline{A\times B}^{\mathcal{G}}\subset \overline{A}^\mathscr{F}\times \overline{B}^{\mathscr{L}} ~~~\text{and}~~~A^{o\mathscr{F}}\times B^{o\mathscr{L}}\subset(A\times B)^{o\mathcal{G}}.$$

     \item[(ii)] $F_x\times L_y\subset G_{(x,y)}$ for all $(x,y)\in X\times Y$, then
    $$\overline{A}^\mathscr{F}\times \overline{B}^{\mathscr{L}}\subset\overline{A\times B}^{\mathcal{G}}~~~\text{and}~~~(A\times B)^{o\mathcal{G}}\subset A^{o\mathscr{F}}\times B^{o\mathscr{L}}.$$

    \item[(iii)] $G_{(x,y)}\neq F_x\times L_y$ for $(x,y)\in X\times Y$, then 
    $$\left(\overline{A\times B}^{\mathcal{G}}\right)\triangle\left(\overline{A}^\mathscr{F}\times \overline{B}^{\mathscr{L}}\right)\neq\emptyset~~~\text{and}~~~\left(A^{o\mathscr{F}}\times B^{o\mathscr{L}}\right)\triangle(A\times B)^{o\mathcal{G}}\neq \emptyset.$$
    \end{itemize}
\end{remark}

We will now define what a rough separable space is, using the idea of rough closure. To do this, we first explain what a rough dense set is.

\begin{definition}
    A subset $A$ in a topological space $(X,\tau)$ is said to be rough dense in $X$ with roughness $\mathscr{F}$ or $\mathscr{F}$-dense provided that $\overline{A}^\mathscr{F}=X$. A topological space $X$ is said to be rough separable with roughness $\mathscr{F}$ or $\mathscr{F}$-separable if $X$ has a countable $\mathscr{F}$-dense subset.
\end{definition}

Since, for any subset $A \subset X$, $\overline{A}\subset\overline{A}^\mathscr{F}$ for any rough family $\mathscr{F}$, it follows that every separable topological space is also $\mathscr{F}$-separable. Nonetheless, the reverse implication does not necessarily hold, as demonstrated by the following example:

\begin{example}
    Let us consider $\mathbb{R}$ with the discrete topology $\tau_{disc}$. Subsequently, $\tau_{disc}$ is not separable. Now consider the rough families $\mathscr{F}$ and $\mathcal{G}$ as follows: $$\mathscr{F}=\left\{[x,x+1]: x\in \mathbb{R}\right\}~~~\text{and}~~~\mathcal{G}=\left\{[x,x+\alpha]: x\in \mathbb{R}~\text{and}~\alpha>0\right\}.$$
    Then, $\mathbb{R}$ is both $\mathscr{F}$-separable and $\mathcal{G}$-separable as $\overline{\mathbb{Z}}^\mathscr{F}=\mathbb{R}$ and $\overline{\mathbb{Q}}^\mathcal{G}=\mathbb{R}$.
\end{example}

In line with the usual concepts of closed and open sets, we introduce the ideas of rough closed sets and rough open sets. The definitions are outlined as follows:

\begin{definition}
Let $\mathscr{F}$ be a rough family in a topological space $(X,\tau)$. A set $A\subset X$ is said to be rough closed with roughness $\mathscr{F}$ or $\mathscr{F}$-closed provided that $\overline{A}^\mathscr{F}=A$ and is called  rough open with roughness $\mathscr{F}$ or $\mathscr{F}$-open provided that $A^{o\mathscr{F}}=A$. 
\end{definition}

\begin{example}\label{Ex3}
    Let us consider the space $\mathbb{R}_l$ and let $A=[0,1).$ Fix the rough families $\mathscr{F}$ and $\mathcal{G}$ as follows: $$\mathscr{F}=\left\{\left\{x,\frac{x}{2}\right\}: x\in \mathbb{R}\right\}~~~\text{and}~~~\mathcal{G}=\left\{\left\{x,2x\right\}: x\in \mathbb{R}\right\}.$$ Then,
    we obtain $$A^{o\mathscr{F}}=[0,1)=A~~~\text{and}~~~\overline{A}^\mathcal{G}=[0,1)=A.$$  
    Hence, the set $A$ is $\mathscr{F}$-open and is $\mathcal{G}$-closed.
\end{example}

\begin{remark}\label{Rem2}
    \begin{itemize}
        \item[(i)]  $\overline{\emptyset}^\mathscr{F}=\emptyset,$ $\emptyset^{o\mathscr{F}}=\emptyset,$  $\overline{X}^\mathscr{F}=X,$ $X^{o\mathscr{F}}=X$ for any rough family $\mathscr{F}$. Therefore, $\emptyset, X$ both are $\mathscr{F}$-closed and $\mathscr{F}$-open for any rough family $\mathscr{F}$.
        
        \item[(ii)] If $A, B$ are $\mathscr{F}$-open sets, so is $A\cap B$. Similarly, if $A, B$ are $\mathscr{F}$-closed sets, so is $A\cup B$.

        \item[(iii)] Building on Theorem \ref{T1}, we can conclude that for any roughness $\mathscr{F}$, the $\mathscr{F}$-open sets are indeed open, and the $\mathscr{F}$-closed sets are closed. However, Example \ref{Ex2} shows us that the reverse isn't always true.

        \item[(iv)] From Example \ref{Ex3}, we get $$\overline{A}^\mathscr{F}=[0,2)\subset [0,4)=\overline{\left(\overline{A}^\mathscr{F}\right)}^{\mathscr{F}}   ~~~\text{and}~~~{\left(A^{o\mathcal{G}}\right)}^{o\mathcal{G}}=\left[0,\frac{1}{4}\right)\subset\left[0,\frac{1}{2}\right)=A^{o\mathcal{G}}.$$    
        Thus, for any set $A\subset X,$ $\overline{A}^\mathscr{F}$ may fail to be  $\mathscr{F}$-closed. Similarly, $A^{o\mathscr{F}}$ may fail to be  $\mathscr{F}$-open. Hence, following Theorem \ref{T1}, it can be concluded that for any rough family $\mathscr{F},$ the operator $A\to \overline{A}^\mathscr{F}$ is an expanding operator and the operator  $A\to A^{o\mathscr{F}}$ is a shrinking operator.
    \end{itemize}
\end{remark}

The subsequent theorem elucidates the relationship between $\mathscr{F}$-closed sets and $\mathscr{F}$-open sets.

\begin{theorem}\label{T6}
    A set $A\subset X$ is $\mathscr{F}$-closed if and only if its complement $X\setminus A$ is $\mathscr{F}$-open.
\end{theorem}

\begin{proof} If $A=\emptyset,$ then it is obvious. So, let $A\subset X$ be a nonempty $\mathscr{F}$-closed set.  Let $x\in X\setminus A.$ Thus $x\notin A=\overline{A}^\mathscr{F}.$ Consequently, there exists an open set $U$ with $F_x\subset U$ but $U\cap A\setminus\{x\}=\emptyset.$ Hence, $U\subset X\setminus A$ and thus $x\in (X\setminus A)^{o\mathscr{F}}$ implies that $X\setminus A$ is $\mathscr{F}$-open.\\

On the other hand, let $A$ be $\mathscr{F}$-open. Let $x\notin X\setminus A$. Then,  $x\in A$ and thus there is an open set $U$ in $X$ with $F_x\subset U\subset A$. Hence, $\left\{U\cap (X\setminus A)\right\}\setminus\{x\}=\emptyset$ which implies that $x\notin \overline{(X\setminus A)}^\mathscr{F}$. Consequently, $\overline{(X\setminus A)}^\mathscr{F}\subset X\setminus A$ and therefore, $X\setminus A$ is $\mathscr{F}$-closed.
\end{proof}

\begin{theorem}\label{Th3.7} Let $\mathscr{F}$ be a rough family in a topological space $(X,\tau)$ and $A\subset X$. Then the following are true: \begin{itemize}
    \item[(i)] If $A$ is $\mathscr{F}$-open set (resp. $\mathscr{F}$-closed) in $(X,\tau)$, then $A$ is also $\mathscr{F}$-open (resp. $\mathscr{F}$-closed) in $(X,\zeta)$ for every $\tau\subset \zeta$. The reverse implication is not necessarily true.

    \item[(ii)] $A$ is $\mathscr{R}$-open (resp. $\mathscr{R}$-closed) implies $A$ is $\mathscr{F}$-open (resp. $\mathscr{F}$-closed) as well for any rough family $\mathscr{R}$  in $X$ with $\mathscr{F}\leq \mathscr{R}$.
\end{itemize}
\end{theorem}

\begin{proof}
\begin{itemize}
    \item[(i)] By assumption, $A^{o\mathscr{F}}=A$ in $(X,\tau)$. Let $x\in A^{o\mathscr{F}}$ in $(X,\tau)$. Then there exists $U\in \tau$ with $F_x\subset U\subset A$. Consequently, $U\in \zeta$ with $F_x\subset U\subset A$ implies that $x\in A^{o\mathscr{F}}$ in $(X,\zeta)$ and thus $A^{o\mathscr{F}}=A$ also in $(X,\zeta)$.

    The proof that $A$ is $\mathscr{F}$-closed in $(X, \zeta)$ whenever it is $\mathscr{F}$-closed in $(X, \tau)$ follows by a similar argument.

    For the subsequent part, let us consider the rough family $\mathscr{F}=\left\{\left\{x,\frac{x}{2}\right\}: x\in \mathbb{R}\right\}$ in both $\mathbb{R}_l$ and $\mathbb{R}_u$. Then the set $A=[0,1)$ is $\mathscr{F}$-open in $\mathbb{R}_l$. However, in $\mathbb{R}_u$, $A^{o\mathscr{F}}=(0,1)\neq A$ implies that $A=[0,1)$ is not $\mathscr{F}$-open in $\mathbb{R}_u$.

    \item[(ii)] This is an immediate consequence of Corollary \ref{T1}. 
\end{itemize}
\end{proof}

A pertinent question arises as to whether the product of two rough closed sets or rough open sets is rough closed or rough open. The following theorem provides a definitive answer to this inquiry.

\begin{theorem}
     Let $\mathscr{F}=\{F_x: x\in X\}$ and $\mathscr{L}=\{L_y: y\in Y\}$ be two rough families in two topological spaces \((X,\tau)\) and \((Y,\zeta)\) respectively. Then for the rough family $\mathcal{G}=\{G_{(x,y)}=F_x\times L_y: F_x\in\mathscr{F}~\text{and}~L_y\in \mathscr{L}\}$ in $X\times Y$, the following hold true:
    \begin{itemize}
        \item[(i)] If $A\subset X$ be $\mathscr{F}$-closed and $B\subset Y$ be $\mathscr{L}$-closed, then $A\times B$ is a $\mathcal{G}$-closed set in $X\times Y$.
        \item[(ii)] If $A\subset X$ be $\mathscr{F}$-open and $B\subset Y$ be $\mathscr{L}$-open, then $A\times B$ is a $\mathcal{G}$-open set in $X\times Y$. 
    \end{itemize}
\end{theorem}

\begin{proof} \begin{itemize}
        \item[(i)]  Due to the assumptions, $\overline{A}^\mathscr{F}=A$ and $\overline{B}^\mathscr{L}=B$. Moreover, using Theorem \ref{Th2.5}, we obtain
$$\overline{A\times B}^\mathcal{G}=\overline{A}^\mathscr{F}\times \overline{B}^\mathscr{L}=A\times B.$$
 \item[(ii)] Proof is similar to (i), so omitted.
    \end{itemize}
\end{proof}

\section{ Rough homeomorphism with roughness $(\mathscr{F},\mathscr{L})$}\label{Sec3}
In this section, our primary aim is to broaden the concept of homeomorphism by incorporating rough families. We will illustrate that certain spaces, although not homeomorphic in the conventional sense, can be roughly homeomorphic under specific roughness conditions. Now a map \(f: (X,\tau)\to (Y,\zeta)\) is said to be continuous if, for any \(U\in \zeta\), \(f^{-1}(U)\in\tau\). Moreover, \(f\) is a homeomorphism if it is bijective, continuous, and its inverse \(f^{-1}: (Y,\zeta)\to (X,\tau)\) is also continuous. Following these ideas, we next introduce the notion of rough continuity with roughness $(\mathscr{F},\mathscr{L})$, and explore how it contrasts with the usual definition of continuity.

\begin{definition}
 Let $\mathscr{F}=\{F_x: x\in X\}$ and $\mathscr{L}=\{L_y: y\in Y\}$ be two rough families in two topological spaces \((X,\tau)\) and \((Y,\zeta)\) respectively. A map $f: X\rightarrow Y$ is  said to be rough continuous with roughness $(\mathscr{F},\mathscr{L})$ on $X$ or $(\mathscr{F},\mathscr{L})$-continuous on $X$ 
 provided that for any $\mathscr{L}$-open sets $V$ in $Y$,  $f^{-1}(V)$ is $\mathscr{F}$-open in $X$.

 Equivalently, we can say $f:X\to Y$ is  $(\mathscr{F},\mathscr{L})$-continuous on $X$, if for each $\mathscr{L}$-closed set $A$ in $Y$, $f^{-1}(A)$ is $\mathscr{F}$-closed in $X.$
\end{definition}

 The following example shows the incongruity between $(\mathscr{F},\mathscr{L})$-continuity and usual continuity:

\begin{example}
\begin{itemize}
    \item[(i)]  Let us consider a self-mapping $f$ on $\mathbb{R}_u$ defined by $$f(x)=\begin{cases}
        0, &x\in\mathbb{Q};\\
        1, &x\in \mathbb{R}\setminus\mathbb{Q}.
    \end{cases}$$
Now, $f^{-1}(-1,1)=\mathbb{Q},$ which is not an open set in $\mathbb{R}$. So, $f$ is not continuous. 

Let us now consider the rough families $\mathscr{F}$ and $\mathscr{L}$ as follows:
$$\mathscr{F}=\{\{x\}:x\in \mathbb{R}\}~~~\text{and}~~~\mathscr{L}=\{\{y\}\cup\{0,1\}: y\in \mathbb{R}\}.$$
It is plain that $f$ is a $(\mathscr{F},\mathscr{L})$-continuous map. 

\item[(ii)] Now, consider the identity map $f$ in $\mathbb{R}_u$. Clearly, $f$ is continuous. Now, let us define the rough families $\mathscr{F}$ and $\mathscr{L}$ as follows:
$$\mathscr{F}=\{\{x,1\}:x\in \mathbb{R}\}~~~\text{and}~~~\mathscr{L}=\{\{y\}: y\in \mathbb{R}\}.$$
Then any $\mathscr{F}$-open set in $\mathbb{R}$ is of the form $(a,b)$ where $a,b\in \mathbb{R}$ with $a<1<b$ whereas any open interval is also $\mathscr{L}$-open. Hence, if $U=(2,3)$, then $U$ is $\mathscr{L}$-open but $f^{-1}(U)=U$ which is not a $\mathscr{F}$-open set. Consequently, $f$ fails to be a $(\mathscr{F},\mathscr{L})$-continuous map.
\end{itemize}
\end{example}

\begin{theorem} Any $(\mathscr{F},\mathscr{L}^\#)$-continuous map is a continuous map. Moreover, any continuous map is $(\mathscr{F^\#},\mathscr{L})$-continuous map. 
\end{theorem}

\begin{proof}
    Let $f: (X,\tau)\to(Y,\zeta)$ be a $(\mathscr{F},\mathscr{L}^\#)$-continuous map and let 
    $V\in \zeta$. Then, $V$ be also a $\mathscr{L}^\#$-open set in $(Y,\zeta)$. Subsequently, $f^{-1}(V)$ is a $\mathscr{F}$-open set and is open in $(X,\tau)$ as well. Hence, the map $f: (X,\tau)\to(Y,\zeta)$ is continuous.

    Again let, $g: (X,\tau)\to(Y,\zeta)$ be a continuous map and let $V$ be a $\mathscr{L}$-open set in $(Y,\zeta)$. Consequently, $V\in \zeta$ implies $g^{-1}(V)\in \tau$ and thus $g^{-1}(V)$ is a $\mathscr{F}^\#$-open set in $(X,\tau)$, which ensures that $g$ is $(\mathscr{F^\#},\mathscr{L})$-continuous.
\end{proof}

\begin{corollary}
\begin{itemize}
    \item[(i)] Any $(\mathscr{F},\mathscr{L})$-continuous map equivalents to a continuous map if and only if $\mathscr{F}:=\mathscr{F^\#}=\{\{\eta\}: \eta\in X\}$ and $\mathscr{L}:=\mathscr{L}^\#=\{\{\xi\}:\xi\in Y\}.$

    \item[(ii)] Any $(\mathscr{F},\mathscr{L}^\#)$-continuous open (or closed) bijective map is a homeomorphism.
\end{itemize}
      
\end{corollary}

\begin{theorem}
 Let $f: X\rightarrow Y$ be $(\mathscr{F},\mathscr{L})$-continuous. Then $f$ is also $(\mathscr{F}_\alpha,\mathscr{L}_\beta)$-continuous, where $\mathscr{F}_\alpha$ and $\mathscr{L}_\beta$ are two rough families in $X$ and $Y$, respectively, with $\mathcal{F_\alpha}\leq \mathscr{F}$ and $\mathscr{L}\leq \mathscr{L}_\beta$.
\end{theorem}

\begin{proof}
   Let $V$ be a $\mathscr{L}_\beta$-open set in $Y$. According to Theorem \ref{Th3.7}, $V$ is also a $\mathscr{L}$-open set within $Y$. Consequently, $f^{-1}(V)$ is a $\mathscr{F}$-open set, which implies that $f^{-1}(V)$ is also a $\mathscr{F}_\alpha$-open set in $X$. Hence, $f$ is $(\mathcal{F_\alpha},\mathscr{L}_\beta)$-continuous as well.
\end{proof}

As expected, the following result demonstrates that rough continuity ensures the preservation of rough convergence:
\begin{theorem} Let $f: X\rightarrow Y$ be $(\mathscr{F},\mathscr{L})$-continuous on $X$. Then \begin{itemize}
    \item[(i)] for any $\{x_n\}_{n\in \mathbb{N}}$ in $X$ with $x_n\xrightarrow[]{\mathscr{F}} x$ implies $f(x_n)\xrightarrow[]{\mathscr{L}} f(x).$ 
    \item[(ii)] for $A\subset X,$ $f\left(\overline{A}^\mathscr{F}\right)\subseteq \overline{f(A)}^\mathscr{L}.$
\end{itemize}
\end{theorem}

 \begin{proof}\begin{itemize}
     \item[(i)] Let $V$ be an $\mathscr{L}$-open set in $Y$ with $L_{f(x)}\subset V$. Since, $f$ is $(\mathscr{F},\mathscr{L})$-continuous, $f^{-1}(V)$ is an $\mathscr{F}$-open set in $X$ and $x\in f^{-1}(V)$. 
    
       Now, $x_n\xrightarrow[]{\mathscr{F}} x$ implies that there exists $k\in \mathbb{N}$ so that $x_n\in f^{-1}(V)$ for all $n\ge k.$ Consequently, $f(x_n)\in V$ for all $n\ge k$ which implies that $f(x_n)\xrightarrow[]{\mathscr{L}} f(x)$. 

       \item[(ii)] Let $x\in \overline{A}^\mathscr{F}$ and let $V$ be any open set in $Y$ with $L_{f(x)}\subset V.$ Since $f$ is rough continuous with roughness $(\mathscr{F},\mathscr{L}),$ $f^{-1}(V)$ is an open set in $X$ with $F_{x}\subset f^{-1}(V).$ Thus, $$f^{-1}(V)\cap A\setminus\{x\}\neq \emptyset~\text{implies}~V\cap f(A)\setminus \{f(x)\}\neq \emptyset.$$ Hence, $f(x)\in \overline{f(A)}^\mathscr{L}$ and consequently,   $f\left(\overline{A}^\mathscr{F}\right)\subseteq \overline{f(A)}^\mathscr{L}.$
 \end{itemize}
 \end{proof}

We shall now introduce the concept of rough homeomorphism characterized by roughness $(\mathscr{F},\mathscr{L})$. The precise definition is as follows:

\begin{definition}
Let $\mathscr{F}=\{F_x: x\in X\}$ and $\mathscr{L}=\{L_y: y\in Y\}$ be two rough families in two topological spaces \((X,\tau)\) and \((Y,\zeta)\) respectively. A map $f: X\rightarrow Y$ is  said to be rough homeomorphism with roughness $(\mathscr{F},\mathscr{L})$ or $(\mathscr{F},\mathscr{L})$-homeomorphism provided that $f$ is both bijective, $(\mathscr{F},\mathscr{L})$-continuous map and $f^{-1}: Y\to X$ is a $(\mathscr{L},\mathscr{F})$-continuous map. In this case, the spaces $X$ and $Y$ are called rough homeomorphic with roughness $(\mathscr{F},\mathscr{L})$ or $(\mathscr{F},\mathscr{L})$-homeomorphic space. 
\end{definition}

Like rough continuity, here also $(\mathscr{F},\mathscr{L})$-homeomorphism becomes equivalent to homeomorphim if and only if $\mathscr{F}:=\mathscr{F^\#}=\{\{\eta\}: \eta\in X\}$ and $\mathscr{L}:=\mathscr{L}^\#=\{\{\xi\}:\xi\in Y\}.$ The following example demonstrates that two topological spaces may be rough homeomorphic without being homeomorphic:
 \begin{example}
     Let us consider the spaces $(X,\tau)$ as $\mathbb{R}_u$ and $(Y,\zeta)$ as $\mathbb{R}_l$. Then $X$ and $Y$ are not homeomorphic.
     
     We consider the rough families $\mathscr{F}$ and $\mathscr{L}$ on $X$ and $Y$ as follows:
    $$\mathscr{F}=\mathscr{L}=\{[-|x|,|x|]:x\in \mathbb{R}\}=\left\{\begin{cases}
        [-x,x], &x\ge 0;\\
         [x,-x], &x<0;
     \end{cases}: x\in \mathbb{R}\right\}.$$
     Let us construct the map $f: X\to Y$ by $f(x)=-x$ for all $x\in X$. Then, $f$ is a bijective map and $f^{-1}: Y\to X$ is defined by $f^{-1}(x)=-x$ for all $x\in Y$.

     Note that, any $\mathscr{L}$-open set in $Y$ is of the form $(\min\{-a,a\},\max\{-a,a\})$ for $a\in Y$. Now, $$f^{-1}(\min\{-a,a\},\max\{-a,a\})=(\min\{-a,a\},\max\{-a,a\})$$
     which is an $\mathscr{F}$-open set in $X$. Hence, $f: X\to Y$ is a $(\mathscr{F},\mathscr{L})$-continuous map.

     Again, any $\mathscr{F}$-open set in $X$ is of the form $(\min\{-b,b\},\max\{-b,b\})$ for $b\in X$. Also,
    $$\left(f^{-1}\right)^{-1}(\min\{-b,b\},\max\{-b,b\})=(\min\{-b,b\},\max\{-b,b\})$$
     which is a $\mathscr{L}$-open set in $Y$. Consequently, the map $f^{-1}: Y\to X$ is a $(\mathscr{L},\mathscr{F})$-continuous map. 
     
     So, combining both cases, we can conclude that $f: X\to Y$ is a rough homeomorphism with roughness $(\mathscr{F},\mathscr{L})$.
 \end{example}

\section{Topology induced by rough open sets}\label{Sec4}

Motivated by the works of Liu \cite{Liu3} and Zhou \cite{Zhou}, in this section, our main objective is to develop novel topologies on a specified topological space $(X,\tau)$ by utilizing a rough family $\mathscr{F}$. Moreover, in this section, we aim to characterize these newly established topological spaces and analyze their relationship to the existing topology.

\begin{definition}
    Let $(X,\tau)$ be a topological space and let $\mathscr{F}$ be a rough family on $X$. 
    Then the collection $\{A\subset X: A^{o\mathscr{F}}=A\}$ induces a topology on $X$, called rough topology, with roughness $\mathscr{F}$, denoted by $\tau^\mathscr{F}$. The pair $(X,\tau^\mathscr{F})$ is called a rough topological space with roughness $\mathscr{F}$ induced by $(X,\tau)$.
\end{definition}

In this context, we present examples to illustrate how the topology $\tau^\mathscr{F}$ diverges from the initial topology $\tau$.
\begin{example}\label{Ex5}
    Consider the space $\mathbb{R}_l$ and the rough family $\mathscr{F}=\left\{\left\{x,\frac{x}{2}\right\}: x\in \mathbb{R}\right\}$ in $\mathbb{R}$. Then, the family of $\mathscr{F}$-open sets belongs to the following collection:
    $$\beta^\mathscr{F}=\left\{(-a,0),(0,a),[-a,0),[0,a), [-a,b),(-a,b): a,b\in \mathbb{R}~\text{with}~a,b>0\right\}.$$  
    Consequently, the collection $\beta^\mathscr{F}$ forms the basis for the topology $\mathbb{R}^\mathscr{F}_l$. Moreover, $[1,2)\in \mathbb{R}_l$ but $[1,2)^{o\mathscr{F}}=\emptyset\neq [1,2)$ implies $[1,2)\notin \mathbb{R}^\mathscr{F}_l$. Hence, we obtain $\mathbb{R}^\mathscr{F}_l\neq \mathbb{R}_l$.
\end{example}

\begin{example}
    Let $X\neq \emptyset$ and let $a\in X$. We consider the fixed point topology $\tau_a=\{A\subset X: A=\emptyset~\text{or}~ a\in A\}$ and discrete topology $\tau_{disc}$ on $X$. We will show that $\tau_a$ is a rough topology induced by $\tau_{disc}$. Construct the rough family $\mathscr{F}$ as $$\mathscr{F}=\{F_x=\{a,x\}: x\in X\}.$$
    Then, we obtain $$\tau_{disc}^\mathscr{F}=\{\emptyset, X\}\cup\{\cup F_x: x\in X\}.$$ 
    Let $A\in \tau_a$ and let $x\in A$. Eventually, $a\in A$ and thus $F_x=\{x,a\}\in \tau_{disc}^{\mathscr{F}}$ and $F_x\subset A$. This implies $\tau_a\subset \tau_{disc}^\mathscr{F}$. For other direction, let $B\in \tau_{disc}^\mathscr{F}$ and let $y\in B$. Then $\{y,a\}\in \tau_a$ and $\{y,a\}\subset B$ imply that $\tau_{disc}^\mathscr{F}\subset \tau_a$. Combining both cases, we get $\tau_a=\tau_{disc}^\mathscr{F}$.
\end{example}

\begin{remark}
\begin{itemize}
    \item[(i)] In accordance with Remark \ref{Rem2}(iii), it can be asserted that the topology $\tau^\mathscr{F}$ is strictly coarser than the generating topology $\tau$ on $X$.
    
    \item[(ii)] Any continuous map $f: (X,\tau)\to (X,\tau^\mathscr{F})$ is rough continuous with roughness $(\mathscr{F^\#},\mathscr{F}).$
\end{itemize}
\end{remark}

\begin{theorem}\label{Th5.4} Let $\mathscr{F}$ be a rough family in a topological space $(X,\tau)$. Then we have the following: \begin{itemize}
    \item[(i)] $\tau^\mathscr{F} \subset \zeta^{\mathscr{F}}$ for any topology $\zeta$ on $X$ with $\tau \subset \zeta$.
    
    \item[(ii)] $\tau^\mathscr{R} \subset \tau^{\mathscr{F}}$ for any rough family $\mathscr{R}$ in $X$ with $\mathscr{F}\leq \mathscr{R}$.
\end{itemize}
\end{theorem}

\begin{proof}
    The proof is a direct consequence of Theorem \ref{Th3.7} and is thus omitted.
\end{proof}

Theorem \ref{Th5.4} allows us to infer that the rough family serves as a stimulator. Let us consider $\tau$ as the discrete topology on a non-empty set $X$. By defining a rough family $\mathscr{F}$, it is possible to derive other topologies on $X$ that do not satisfy the $T_i~(i=0,1,2,3,4)$ separation axioms. Notably, if $\mathscr{F}=\{F_x=X: x\in X\}$, then $\tau^\mathscr{F}$ becomes the indiscrete topology on $X$. Consequently, the rough family $\mathscr{F}$ enables the transition between finer and coarser topologies on a non-empty set $X$.

 \begin{theorem}\label{L1}
      Let $(X,\tau)$ be a first countable space, and let the elements of $\mathscr{F}$ be open sets. Then $(X,\tau^\mathscr{F})$ is also first countable. But $(X,\tau^\mathscr{F})$ may not be second countable even if $(X,\tau)$ is so.
 \end{theorem}

  \begin{proof}
      Let $x\in X.$ Then, there is a countable local basis $\{\beta^x_n\}$ in $(X,\tau)$ containing $x$. Now, consider the collection $$\mathscr{F}(\beta^x_n)=\beta^x_n\cup F_x\ \text{ for all }n.$$
      
      Let $U$ be a $\mathscr{F}$-open set in $X$. Then, $U\in \tau$. So, there exists a $\beta_N^x\in \{\beta^x_n\}$ such that $x\in \beta^x_N\subset U$. Also, $F_x\subset U$. Hence, $\beta^x_N\cup F_x\subset U$ and $\beta^x_N\cup F_x\in \{\mathscr{F}(\beta^x_n)\}$.
      Consequently, the collection $\{\mathscr{F}(\beta^x_n)\}$ forms a countable basis at $x$ in $(X,\tau^\mathscr{F})$.

      For the subsequent part, we consider the rough family $\mathscr{F}$ in the second countable space $\mathbb{R}_u$ as $$\mathscr{F}=\{F_x=\{x,2|x|\}: x\in \mathbb{R}\}.$$
      Then $\mathbb{R}^\mathscr{F}_u=\{(a,\infty):  a\in \mathbb{R}\},$ which is not a second countable space.
  \end{proof}

\begin{theorem}\label{Th5.1}
   For any sequence $\{x_n\}_{n\in \mathbb{N}}$ and any $x$ in $X,$ $x_n\xrightarrow[]{\mathscr{F}} x$ in $(X,\tau)$ implies $x_n\xrightarrow[]{} x$ in $\left(X,\tau^\mathscr{F}\right)$.  The converse may not always be true.
\end{theorem}

\begin{proof}
    Let $U$ be any open set in $\left(X,\tau^\mathscr{F}\right)$ containing $x.$ Then, there exists $V\subset X$ such that $x\in V\subset U$ and $V^{o\mathscr{F}}=V$. Hence, $x\in F_x\subset V\subset U$ and $V\in \tau$ as well. Now $x_n\xrightarrow[]{\mathscr{F}} x$ in $(X,\tau)$ implies that there exists $k\in\mathbb{N}$ so that for all $n\ge k,$ $x_n\in V.$ Thus $x_n\in U$ for all $n\ge k$ implies $x_n\xrightarrow[]{} x$ in $\left(X,\tau^\mathscr{F}\right).$\\

   For the converse part, we revisit Example \ref{Ex5}. Let us consider the sequence $\{x_n\}_{n\in \mathbb{N}}$ presented below $$x_n=\begin{cases}\frac{1}{4},&\text{if}~n~\text{is even};\\
     1, &\text{if}~n~\text{is odd}.
     \end{cases}$$
     Then, for any open set $U$ in $\mathbb{R}^\mathscr{F}_l$ containg $1$ is one of the following form $$(0,b)~\text{or}~[0,b)~\text{or}~(-a,b)~\text{or}~[-a,b)~\text{where}~a,b\in \mathbb{R}~\text{with}~0<a~\text{and}~1<b.$$
     In all the cases, $x_n\in U$ for all $n\in \mathbb{N}.$ Hence, $x_n\xrightarrow[]{} 1$ in $\left(X,\tau^\mathscr{F}\right).$ Moreover, for any $x\ge 1$, $x_n\xrightarrow[]{} x$ in $\left(X,\tau^\mathscr{F}\right).$ 

     Now, $V=\left[\frac{1}{2},2\right)$ is an open set in $\mathbb{R}_l$ with $F_1\subset V$. But for every $n\in 2\mathbb{N}$, $x_n\notin V$. Hence, $x_n\not\xrightarrow[]{\mathscr{F}} x$ and we are done.
\end{proof}

The second part of Theorem \ref{Th5.1} implies that the limit $x_n\to x$, where $x\in X$ is not necessarily unique.  Consequently, the space $\left(X,\tau^\mathscr{F}\right)$ need not be a $T_2$-space. This naturally raises the question of whether the space $(X,\tau^\mathscr{F})$ can ever be a $T_2$-space? The next theorem provides a negative answer.

\begin{theorem}
     The space $\left(X,\tau^\mathscr{F}\right)$ is not a $T_2$-space provided that $\mathscr{F} \neq \mathscr{F^\#}$.
\end{theorem}
\begin{proof}
    Since  $\mathscr{F} \neq \mathscr{F^\#}$, there exists $z\in X$ so that $a\in F_z$ where $a\neq z$. Consider the sequence $x_n=a$ for all $n\in \mathbb{N}$. Consequently, we get  $x_n\rightarrow a$. Now, for any $U\in \tau^{\mathscr{F}}$ that contains $z$ must contain a $V\subset X$ where $z\in V$ and ${V}^{o\mathscr{F}}=V$. This implies $z\in F_z\subset V\subset U$ and thus $x_n\rightarrow z$ as well. 
\end{proof}

\begin{corollary}
     The space $\left(X,\tau^\mathscr{F}\right)$ can never be a $T_3 ~\&~T_4$-space and thus $\left(X,\tau^\mathscr{F}\right)$ is not a metrizable space for any topology $\tau$ on $X$ and for any rough family $\mathscr{F} \neq \mathscr{F^\#}$.
\end{corollary}

Now, what about the converse part, i.e., for any non-$T_2$-space $(X,\zeta)$, whether there exists another topology $\tau$ and a rough family $\mathscr{F}$ on $X$, so that $\tau^\mathscr{F}=\zeta$. In the following theorem, we give an answer to this query:

\begin{theorem}
    A topological space $(X,\tau)$ is not $T_2$ if and only if there exists a rough family $\mathscr{F}$ on $X$ such that $\tau=\tau_{disc}^\mathscr{F}$, where $\tau_{disc}$ is the discrete topology on $X$.
\end{theorem}
\begin{proof} One part is a direct consequence of the previous theorem. So we proceed to the other part. Since $(X,\tau)$ is not a $T_2$-space, there exists a net $\{x_i\}_{i\in I}$ in $X$ that converges more than one element in $X$.

    Define a relation $\sim$ on $X$ as $x\sim y$ if there exists a net $\{x_i\}_{i\in I}$ in $X$ such that $x_i\rightarrow x\iff x_i\rightarrow y$. Subsequently, it forms an equivalence relation on $X$. Construct the rough family $\mathscr{F}$ as \[\mathscr{F}=\{F_x=\text{equivalence class of }x: x\in X\}.\]
    Then in $(X,\tau_{disc})$, the collection $$\{A\subset X: A^{o\mathscr{F}}=A\}=\{\emptyset, X\}\cup\{F_x: x\in X\},$$ and thus $$\tau_{disc}^\mathscr{F}=\{\emptyset, X\}\cup\{\cup F_x: x\in X\}.$$ Now our aim is to prove $\tau\subset \tau_{disc}^\mathscr{F}$. Let $A\in \tau$. If $A=\emptyset\text{ or }X$, then there is nothing to prove. So, let $A\subsetneq X$. Now for $a\in A$, either $F_a=\{a\}\subset A$ or if $b\in F_a$, then $b\in A$. Otherwise, the net $x_i=a$ for all $i\in I$ does not converge to $b$, which contradicts the construction of $F_a$. Hence, we have $\tau \subset \tau_{disc}^\mathscr{F}$. Similarly, we can prove that $\tau_{disc}^\mathscr{F}\subset \tau$ and thus $\tau=\tau_{disc}^\mathscr{F}$. 
\end{proof}

\begin{corollary}
    A topological space $(X,\tau)$ is $T_2$ if and only if there is no rough family $\mathscr{F}$ other than $\mathscr{F}^\#$ such that $\tau=\tau_{disc}^\mathscr{F}$, where $\tau_{disc}$ is the discrete topology in $X$.
\end{corollary}

\begin{question}
     For any non-$T_2$ topology $\zeta$, and any $T_2$ topology $\tau~
     (\neq \tau_{disc}$) on $X$,  is it possible to construct a rough family $\mathscr{F}$ on $X$ such that $\tau^{\mathscr{F}}=\zeta$?
\end{question}

\begin{theorem}\label{T8}
 If $(X,\tau)$ is compact (resp. connected), then $(X,\tau^\mathscr{F})$ is compact (resp. connected). The converse may not be true.
\end{theorem} 

\begin{proof}
    The proof is straightforward, as every $\mathscr{F}$-open set is an open set. 

    For the subsequent parts, we shall first focus on the space $\mathbb{N}$ equipped with the discrete topology $\tau_{disc}$. Consider the rough family $$\mathscr{F}=\{F_n: n\in \mathbb{N}\}~\text{where}~F_n=\{k\in \mathbb{N}: k\ge n\}.$$
    Then any $\mathscr{F}$-open set is of the form $A_n=\{k\in \mathbb{N}: k\ge n\}.$ Hence, the topology $\tau_{disc}^\mathscr{F}$ on $\mathbb{N}$ becomes compact whereas $(\mathbb{N},\tau_{disc})$ is a non-compact space.\\

    For another part, in $\mathbb{R}_l$, we construct the rough families $\mathcal{G}$ as $\mathcal{G}=\left\{\{x,2|x|\}: x\in \mathbb{R}\right\}.$
    Then, the $\mathcal{G}$-open sets are of the form $[a,\infty),$ $a\in\mathbb{R},$ that induces the topology $\mathbb{R}^{\mathcal{G}}_l$. Clearly, the topology $\mathbb{R}^{\mathcal{G}}_l$ is connected since there do not exist two disjoint open sets, whereas $\mathbb{R}_l$ is not connected. 
\end{proof}

Now, a natural inquiry emerges: Can we establish a criterion for the compactness or connectedness of $(X,\tau^\mathscr{F})$ that is independent of the compactness or connectedness of the initial space $(X,\tau)$? The subsequent theorem provides an affirmative answer to this inquiry.

\begin{theorem}\label{Th5.3}
 Let $(X,\tau)$ be a topological space and let $\mathscr{F}$ be a rough family in $X$. Then the induced space $(X,\tau^\mathscr{F})$ is 
 \begin{itemize}
    \item[(i)] compact if and only if every family of $\mathscr{F}$-closed subsets in $(X,\tau)$ that has the finite intersection property (F.I.P.) has a nonempty intersection;

    \item[(ii)] connected if and only if for any subset $A$ of $X$ that is both $\mathscr{F}$-open and $\mathscr{F}$-closed in $(X,\tau)$ implies either $A=\emptyset$ or $A=X$.
\end{itemize}
\end{theorem}

\begin{proof}
     The proof of the theorem follows the classical approach, relying on Theorem~\ref{T6}.
\end{proof}

Determining whether two topological spaces are homeomorphic often presents significant challenges. We provide an alternative approach using the concepts of rough homeomorphism and rough topology.

\begin{theorem}
    Let $(X,\tau)$ and $(Y,\zeta)$ be $(\mathscr{F},\mathscr{L})$-homeomorphic. Then $(X,\tau^\mathscr{F})$ and $(Y,\zeta^\mathscr{L})$ are homeomorphic.
\end{theorem}

\begin{proof}
    By assumption, there is a $(\mathscr{F},\mathscr{L})$-homeomorphism $f: (X,\tau)\to (Y,\zeta)$. We assert that $f: (X,\tau^\mathscr{F})\to (Y,\zeta^\mathscr{L})$ is a homeomorphism. To this aim, let $V\in \zeta^\mathscr{L}$. Thus, $V$ be a $\mathscr{L}$-open set in $(Y,\zeta)$. Due to the characteristics of $f$, $f^{-1}(V)$ is $\mathscr{F}$-open set in $(X,\tau)$ and thus $f^{-1}(V)\in \tau^\mathscr{F}$. Hence, $f: (X,\tau^\mathscr{F})\to (Y,\zeta^\mathscr{L})$ is continuous. Similarly, $f^{-1}: (Y,\zeta^\mathscr{L})\to (X,\tau^\mathscr{F})$ is continuous as well. Thus, we conclude that the spaces $(X,\tau^\mathscr{F})$ and $(Y,\zeta^\mathscr{L})$ are indeed homeomorphic.
\end{proof}

\section{Rough compactness \& Rough connectedness with roughness $\mathscr{F}$}\label{Sec5}

One of the significant research directions of this study is the introduction of the concept of compactness and connectedness concerning a rough family $\mathscr{F}$. In this section, we propose the notion of rough compactness and rough connectedness with roughness $\mathscr{F}$ and examine its relationship with the conventional notion of compactness and connectedness. Recall that a topological space is compact if every open cover of $X$ has a finite subcover and is called connected if it cannot be partitioned into two disjoint, non-empty open sets.  Analogously, we define rough compactness and rough connectedness with roughness $\mathscr{F}$ as follows:
\begin{definition}
    Let $\mathscr{F}$ be a rough family in a topological space $(X,\tau)$. Then $(X,\tau)$ is said to be 
    
    \begin{itemize}
        \item rough compact with roughness $\mathscr{F}$ or $\mathscr{F}$-compact provided for every $\mathscr{F}$-open cover has a finite  $\mathscr{F}$-open subcover. A subset $A$ of $X$ is said to be $\mathscr{F}$-compact if for every family $\{U_\alpha\}_{\alpha\in J}$ of $\mathscr{F}$-open subsets of $X$  there exists a finite subset $\{\alpha_1,\alpha_2,\cdot\cdot\cdot,\alpha_k\}\subset J$ so that $A\subset\cup^{k}_{i=1}U_{\alpha_i}.$
  
    \item rough connected with roughness $\mathscr{F}$ or $\mathscr{F}$-connected provided that $X$ cannot be expressed as a union of two disjoint non-void $\mathscr{F}$-open sets.
    \end{itemize}\end{definition}

\begin{remark}
\begin{itemize}
    \item[(i)] Since every $\mathscr{F}$-open set is open for any rough family $\mathscr{F}$, it follows that every compact (resp. connected) topological space is also $\mathscr{F}$-compact (resp. $\mathscr{F}$-connected) for any rough family $\mathscr{F}$. 
    However, the reverse implications are not assured, as substantiated by Theorem \ref{T8}.

    \item[(ii)] A rough closed subset of a rough compact space is also rough compact. Nonetheless, it is crucial to recognize that not all closed subsets of a rough compact space are necessarily rough compact. Consider Example \ref{Ex5}, where $[2,4]$ is $\mathscr{F}$-compact. Now, the subset $[2,3)$ is closed; however, it does not meet the criteria for $\mathscr{F}$-compactness, as the collection $\left\{\left[0,3-\frac{1}{n}\right): n\in \mathbb{N}\right\}$ forms a $\mathscr{F}$-open cover for $[2,3)$ but lacks a finite subcover.
\end{itemize}
   
\end{remark}

It is well established that continuity maintains compactness (or connectedness). However, the subsequent example illustrates that rough compactness is not always preserved under a continuous mapping.

\begin{example}
    Let us construct the map $f: \mathbb{R}_l\to \mathbb{R}_u$ defined by $$f(x)=\frac{1}{1+|x|}, x\in \mathbb{R}.$$
    Clearly, the map $f$ is continuous. Next, consider the rough families $\mathscr{F}$ and $\mathscr{L}$ on $\mathbb{R}_l$ and $\mathbb{R}_u$ respectively, defined as follows: 
    $$\mathscr{F}=\mathscr{L}=\left\{\left\{x,\frac{x}{2}\right\}:x\in \mathbb{R}\right\}.$$ Consequently, the set $A=(0,4]$ becomes a $\mathscr{F}$-compact set in $\mathbb{R}_l$. Any $\mathscr{L}$-open set in $\mathbb{R}_u$ is one of the following form 
    $$(0,b)~\text{or}~(-a,0)~\text{or}~(-a,b)~\text{where}~a,b\in \mathbb{R}~\text{with}~a,b>0.$$
    Hence, the set $f(A)=\left[\frac{1}{5},1\right)$ is not $\mathscr{F}$-compact as the open cover $\left\{\left(0,1-\frac{1}{n}\right):n\in\mathbb{N}\right\}$  does not possess a finite subcover for $\left[\frac{1}{5},1\right)$.
\end{example}

\begin{theorem}\label{Thm6.1}
     Let $f: X\rightarrow Y$ be $(\mathscr{F},\mathscr{L})$-continuous map. Then \begin{itemize}
         \item[(i)] if $X$ is $\mathscr{F}$-compact, then $f(X)$ is $\mathscr{L}$-compact;

         \item[(ii)] if $X$ is $\mathscr{F}$-connected, then $f(X)$ is $\mathscr{L}$-connected.
     \end{itemize}
\end{theorem}

\begin{proof} The proof for $\mathscr{L}$-compactness and $\mathscr{L}$-connectedness are structurally identical; we explain only $\mathscr{L}$-compactness here.

    Let $\{U_\alpha\}_{\alpha\in J}$ be a family of $\mathscr{L}$-open subsets in $Y$ which covers $f(X)$. Then due to $(\mathscr{F},\mathscr{L})$-continuity of $f$, $f^{-1}(U_\alpha)$ is $\mathscr{F}$-open in $X$ for all $\alpha\in J$. Hence, there is a finite subset $\{\alpha_1,\alpha_2,\cdot\cdot\cdot,\alpha_k\}\subset J$ such that $X=\bigcup^{k}_{i=1}f^{-1}(U_{\alpha_i})$ and consequently, $f(X)\subset \bigcup^{k}_{i=1}U_{\alpha_i}$.
\end{proof}

\begin{corollary}
     Let $(X,\tau)$ and $(Y,\zeta)$ be $(\mathscr{F},\mathscr{L})$-homeomorphic. Then $(X,\tau)$ is $\mathscr{F}$-compact (or $\mathscr{F}$-connected) if and only if $(Y,\zeta)$ $\mathscr{L}$-compact (or $\mathscr{L}$-connected).
\end{corollary}

It is natural to seek a characterization of rough compactness of some space through the cluster point notion of a net. To this end, we provide one answer, beginning with the following definition.

\begin{definition}
    Let $S: D\to X$ be a net. Then $x\in X$ is said to be rough cluster point with roughness $\mathscr{F}$ or $\mathscr{F}$-cluster point of $S$ provided for any open set $U$ in $X$ with $F_x\subset U$ and for any $\alpha\in D$, there is $\beta\in D$ with $\beta\geq \alpha$ and $S_\beta\in U$.
    \end{definition}

\begin{theorem}
 If every net in $X$ has a $\mathscr{F}$-cluster point, then $(X,\tau)$ is $\mathscr{F}$-compact.
\end{theorem}

\begin{proof}
Let $\mathscr{C}$ be a family of $\mathscr{F}$-closed sets in $(X,\tau)$ having \textbf{F.I.P.} and $\mathscr{D}$ be the collection of all finite intersections of members of $\mathscr{C}$. Consequently, $\mathscr{D}$ forms a directed set w.r.t. "$\geq$" where $A\geq B$ if and only if $A\subset B$. Consider the net $S: \mathscr{D}\to X$ by $S(A)= \text{an element of } A$. By assumption, $S$ has a  $\mathscr{F}$-cluster point in $X$ say $z$. Our claim is $z\in \bigcap \mathscr{C}$. If not, then $z\in X\setminus A$ for some $A\in \mathscr{C}\subset \mathscr{D}$. Now thanks to Theorem \ref{T6}, $X\setminus A$ is $\mathscr{F}$-open and thus $F_z\subset X\setminus A$. Subsequently, there exists $B\in \mathscr{D}$ with $B\geq A$ and $S(B)\in X\setminus A$ contradicts the fact that $S(B)\in B\subset A$. Hence, $\bigcap \mathscr{C}$ is a non-void set and $(X,\tau)$ becomes compact according to Theorem \ref{Th5.3}.
\end{proof}

Our next goal is to identify the necessary and sufficient conditions for rough compactness (resp. rough connectedness), as detailed in the following theorem:

\begin{theorem}\label{Thm6.2}
    For any rough family $\mathscr{F}$, a topological space $(X,\tau)$ is $\mathscr{F}$-compact (resp. $\mathscr{F}$-connected) if and only if $(X,\tau^\mathscr{F})$ is compact (resp. connected).
\end{theorem}

\begin{proof} Let $(X,\tau)$ be $\mathscr{F}$-compact but $(X,\tau^\mathscr{F})$ fails to be compact. Then, there exists an open cover $\{U_\alpha\}_{\alpha\in J}$ in $(X,\tau^\mathscr{F})$ that has no finite subcover for $X$. Hence, we obtain a family $\{U_\alpha\}_{\alpha\in J}$ of $\mathscr{F}$-open sets in $(X,\tau)$ that has no finite subcover for $X$, contradicts the assumption. So, the space $(X,\tau^\mathscr{F})$ must be compact. The converse part is quite similar, hence omitted.
\end{proof}

In a $T_2$ topological space, a compact subset is inherently closed. However, this does not straightforwardly extend to the concept of rough compactness. Specifically, a $\mathscr{F}$-compact subset in a $T_2$ topological space might not be $\mathscr{F}$-closed, when considering the associated roughness $\mathscr{F}$. To illustrate, consider Example \ref{Ex5}, where the set $[2,4]$ is a $\mathscr{F}$-compact set, yet it is not a $\mathscr{F}$-closed because $\mathbb{R}\setminus[2,4]=(-\infty, 2)\cup (4,\infty)$ fails to be a $\mathscr{F}$-open set.\\

For a rough family $\mathscr{F}$ in a topological space $(X,\tau)$, $X$ is defined as a $\mathscr{F}$-$T_{2}$-space if any two distinct points in $X$ can be separated by two disjoint $\mathscr{F}$-open sets. Consequently, every $\mathscr{F}$-$T_{2}$-space inherently qualifies as a $T_2$-space; however, the reverse implication does not necessarily hold. This is exemplified in Example \ref{Ex5}, where points $2$ and $4$ cannot be separated by two disjoint $\mathscr{F}$-open sets.

\begin{theorem}
 In a $\mathscr{F}$-$T_2$-topological space, a $\mathscr{F}$-compact subset is always $\mathscr{F}$-closed. Hence, closed as well.
\end{theorem}

\begin{proof}
Let $A$ be a $\mathscr{F}$-compact subset of a $\mathscr{F}$-$T_2$-topological space $(X,\tau)$ and let $x\in X\setminus A$. Then for any  $y\in A$, there are two disjoint $\mathscr{F}$-open sets $U_y, V_y$ in $X$ such that $x\in U_y$ and $y\in V_y$. Now consider the collection $\{V_y: y\in A\}$. This will form a $\mathscr{F}$-open cover of $A$ and thus has a finite subcover $\{V_{y_i}: i=1,2,\cdot\cdot\cdot, k\}$ i.e., $A\subset \bigcup^k_{i=1} V_{y_i}$. Let $U_{y_i},~ i=1,2,\cdot\cdot\cdot, k$ be the corresponding $\mathscr{F}$-open sets containing $x$. Let $U=\bigcap^k_{i=1} U_{y_i}$. Then $U$ is also a $\mathscr{F}$-open set and $x\in F_x\subset U\subset X\setminus A$. Hence, $X\setminus A$ is $\mathscr{F}$-open or equivalently, $A$ is $\mathscr{F}$-closed and thus is closed also.
\end{proof}

The paper concludes by examining the concepts of $\mathscr{F}$‑compactness and $\mathscr{F}$‑connectedness, analyzed across various topological contexts and within distinct rough families:
\begin{theorem}\label{Th6.5} Let $\mathscr{F}$ be a rough family in a topological space $(X,\tau)$. Then \begin{itemize}
    \item[(i)] $(X,\tau)$ is $\mathscr{F}$-compact (resp. $\mathscr{F}$-connected) implies that $(X,\zeta)$ is also $\mathscr{F}$-compact (resp. $\mathscr{F}$-connected) for any topology $\zeta$ on $X$ with $\zeta\subset \tau$.
    
    \item[(ii)] if $(X,\tau)$ is $\mathscr{F}$-compact (resp. $\mathscr{F}$-connected), then it follows that $(X,\tau)$ is also $\mathscr{R}$-compact (resp. $\mathscr{R}$-connected ) for any rough family $\mathscr{R}$ in $X$ with $\mathscr{F}\leq \mathscr{R}$.
\end{itemize}
\end{theorem}

\begin{proof}\begin{itemize}
    \item[(i)] Assume that the space $(X,\zeta)$ is not $\mathscr{F}$-compact. Then there exists a collection $\mathcal{O}$ of $\mathscr{F}$-open sets in $(X,\zeta)$ that lacks a finite subcover for $X$. Theorem \ref{Th3.7} asserts that each member of $\mathcal{O}$ is also a $\mathscr{F}$-open set in $(X,\tau)$, which implies that $(X,\tau)$ is not $\mathscr{F}$-compact, leading to a contradiction.

    \item[(ii)] The proof of (ii) is analogous to the proof of (i).
\end{itemize}
    
\end{proof}


\begin{thebibliography}{99}

	\bibitem{Aytar2008(2)} Aytar, S., 2008. The rough limit set and the core of a real sequence. Numer. Funct. Anal. Optim. 29(3-4), 283-290.
	
	\bibitem{Aytar2008(1)} Aytar, S., 2008. Rough statistical convergence. Numer. Funct. Anal. Optim. 29(3-4), 291-303.
	
	\bibitem{balcerzak2019} Balcerzak, M., Leonetti, P.,  2019. On the relationship between ideal cluster points and ideal limit points. Topology Appl. 252, 178-190.
	
	\bibitem{Das}  Das, P., 2012. Some further results on ideal convergence in topological spaces. Topology Appl. 159(10-11), 2621-2626.
	
	\bibitem{SGSM1} Ghosal, S., Mandal, S., 2022. The degree of roughness. Topology Appl. 307, 107944.
	
	\bibitem{SGSM2} Ghosal, S., Mandal, S., 2022. Rough weighted $\mathcal{I}$-$\alpha\beta$-statistical convergence in locally solid Riesz spaces. J. Math. Anal. Appl. 506(2), 125681.  
	

	
	
    \bibitem{Leonetti2023} Leonetti, P., 2025. Rough Families, Cluster Points, and Cores. J. Convex Anal. 32(4), 1083-1090.

    \bibitem{Leonetti2026} Leonetti, P. (2026). On rough ideal convergence. arXiv preprint arXiv:2601.13805.

  
    \bibitem{Liu2} Liu, X., Lin, S., Zhou, X., 2025. The first-countability in generalizations of topological groups with ideal convergence. Topology Appl. 360, 109150.

    \bibitem{Liu3} Liu, L., Lin, S., Zhou, X., 2025. On $G_s$-convergence and $G_s$-countable compactness. Topology Appl. 375, 109554.
	
	\bibitem{Pal2013} Pal, S. K., Ch, D., Dutta, S., 2013. Rough ideal convergence. Hacet. J. Math. Stat. 42(6), 633-640.

	
	\bibitem{Phu2001} Phu, H. X., 2001. Rough convergence in normed linear spaces. Numer. Funct. Anal. Optim. 22(1-2), 199-222.
	
	\bibitem{Phu2003} Phu, H. X., 2003. Rough convergence in infinite dimensional normed spaces. Numer. Funct. Anal. Optim. 24 (3-4), 285-301.

    \bibitem{Rahaman} Rahaman, S. A.,  Mursaleen, M., 2024. On rough deferred statistical convergence of difference sequences in L-fuzzy normed spaces. J. Math. Anal. Appl. 530(2), 127684.


	\bibitem{Zhou} Zhou, X., Liu, L., Lin, S., 2020. On topological spaces defined by I-convergence. Bull. Iranian Math. Soc. 46(3), 675-692.
\end{thebibliography}
\end{document}